\documentclass[twoside,reqno,10pt]{amsart}
	
	\newcommand{\diff}[4]{\ensuremath{ \mathrm{d}^{#4 }_{ #3 } \left[ #1 \right] \left(  #2 \right)  }}
	\newcommand{\diffplus}[3]{\diff{#1}{#2}{#3}{+} }
	\newcommand{\diffmin}[3]{\diff{#1}{#2}{#3}{-} }

	\newcommand{\fdiff}[4]{\ensuremath{ \upsilon^{#4 }_{ #3 }  #1  \left(  #2 \right)  }}
	\newcommand{\fdiffplus}[3]{ \fdiff {#1}{#2}{+}{#3} }
	\newcommand{\fdiffmin}[3]{ \fdiff {#1}{#2}{-}{#3} }
	\newcommand{\fdiffpm}[3]{ \fdiff {#1}{#2}{\pm}{#3} }
	 \newcommand{\fracvar}[4]{\ensuremath{ \upsilon_{ #3 }^{ #4} \left[ #1 \right] \left(  #2 \right)   }}
	 \newcommand{\fracvarplus}[3]{ \fracvar {#1}{#2}{#3}{\epsilon+} }
	 \newcommand{\fracvarmin}[3]{ \fracvar {#1}{#2}{#3}{\epsilon -} }
	 
	\newcommand{\llim}[3]{\ensuremath{ \lim\limits_{ #1 \rightarrow #2} #3 }}
	\newcommand{\fclass}[2]{\ensuremath{  \mathbb{#1}^{\, #2} }}
	
		\newcommand{\holder}[1]{\fclass{H}{#1} }
	
	 \newcommand{\osc}[4]{\ensuremath{ \mathrm{osc}_{ #3 }^{ #4} [ #1 ] \left(  #2 \right)   }}
	 \newcommand{\oscplus}[2]{ \osc {#1}{#2}{\epsilon}{+} }
	\newcommand{\oscmin}[2]{ \osc {#1}{#2}{\epsilon}{-} }
	
	\newcommand{\deltaop}[4]{\ensuremath{ \Delta_{ #3 }^{ #4} \left[ #1 \right] \left(  #2 \right)   }}
	\newcommand{\deltaplus}[2]{ \deltaop {#1}{#2}{\epsilon}{+} }
	\newcommand{\deltamin}[2]{ \deltaop {#1}{#2}{\epsilon}{-} }

	\newcommand{\epnt}{\; .}
	\newcommand{\ecma}{\; ,}

 \newtheorem{theorem}{Theorem}
 \newtheorem{lemma}{Lemma}
 \newtheorem{corollary}{Corollary}
 \newtheorem{proposition}{Proposition}
 
\newtheorem{definition}{Definition}
 \newtheorem{remark}{Remark}
 \newtheorem{example}{Example}

 \usepackage{graphicx,color,psfrag}
 
 \usepackage{lineno}

   \title[Product rules for Fractional variation]
   {Product rules for the Fractional variation and velocity of H\"olderian functions}

 	\author{ Dimiter Prodanov }
    \address[Dimiter Prodanov]{ 
    	Department of Environment, Health and Safety  \\
    	Neuroscience Research Flanders  \\
    	IMEC, Leuven, Belgium 
    }
 	
\begin{document}

 \begin{abstract}
	 Fractional velocity is defined as the limit of the difference quotient of the increments of a function and its argument raised to a fractional power.
	 Fractional velocity can be suitable for characterizing singular behavior of derivatives of H\"olderian functions and non differentiable functions. 
	 The manuscript derives the product rules for fractional variation.  
	 Correspondence with integer-order derivatives is discussed.
	 It is demonstrated that for H\"older functions under certain conditions the product rules deviates from the Leibniz rule. This deviation is expressed by another quantity, fractional co-variation. 
	 Basic algebraic properties of the fractional co-variation are demonstrated. 		 
 			
 	\medskip
 			
 			
 			\smallskip
 			
 \end{abstract}

 		
    
      	  
 \maketitle

   	
   	\section{Introduction}
   	\label{seq:intro}
   	Research on fractional derivatives and fractional calculus has long history  \cite{Ross1977}. 
   	Surprisingly, until the end of the 20\textsuperscript{th} century only non-local definitions of fractional derivatives have been investigated. 
   	
   	Applications to physical systems exhibiting fractal behavior have  inspired the development of local definitions of fractional derivatives.
   	Historically, the first non-integral order local definition has been proposed by Cherbit as the notion of \textsl{$\alpha$-velocity} \cite{Cherbit1991}. It was defined as a differential quotient of the increment of the particle displacement and the fractional power of the time interval between measurements.
   	The concept implicitly implied Holder-continuity of the studied trajectory. 
   	The later definition of local fractional derivative due to Kolwankar and Gangal \cite{Kolwankar1997a} was based on localization of Riemann-Liouville fractional derivatives, however this derivative is difficult to calculate for a wide range of functions. 
 
    The correspondences between the integral and the quotient difference definitions have been recently investigated by Chen et al. \cite{Chen2010}.
   	On the other hand,  Ben Adda and Cresson demonstrated equivalence between the $\alpha$-velocity and the local fractional derivatives in the sense of Kolwankar and Gangal \cite{Adda2001, Adda2013}.    	
  	 
   	The present work uses the definition of \textsl{$\alpha$-fractional velocity} of functions introduced by Cherebit. Results for \textsl{$\alpha$-velocity} in this work are based on the previously introduced notion of fractional variation operators \cite{Prodanov2015}.    	
   	This manuscript demonstrates the product rule for \textsl{$\alpha$-velocity} and establishes its relationship with the classical Leibniz's product rule. 
   	
   	The main results can be stated in the following pair of equations
	   \[
   		\fdiffpm{ [f \, g]  }{x}{\alpha}    =   \fdiffpm{ f}{x}{\alpha}  g (x) +   \fdiffpm{ g}{x}{\alpha} f(x) 	\pm	[f,g]^{\pm}_\alpha (x) \ecma \\
	  \]
   	where in turn $\fdiffpm{ f}{x}{\alpha} $ designates the  \textsl{$\alpha$-fractional velocity} (or in short fractional velocity or {$\alpha$-velocity when discussing a particular order) of the function $f(x)$ of order $\alpha$ and
   	$  [f,g]^{\pm}_{\alpha} $ designates the fractional  co-variation(s) of the functions  $f(x)$ and $g(x)$ (of order $\alpha$).
   	\textsl{Fractional velocity} can be especially suitable for characterizing of singular behavior of derivatives of H\"olderian functions.
   	Such functions are considered, for example, in quantum mechanical systems \cite{Abbott1981} or in the physical theory of scale relativity \cite{Nottale1989, Nottale1998, Nottale2011} where geodesic trajectories are considered to be non-differentiable. 
   	
   	The present manuscript is organized as follows.
   	Section \ref{sec:definitions} gives general definitions and notational conventions.
   	Section \ref{sec:frdiff} introduces fractional variation operators and $\alpha$-fractional velocities.
   	Section \ref{sec:cofrdiff} introduces fractional co-variation operators.
   	Section \ref{sec:prodrule} proves the main result and demonstrates the conditions under which the usual Leibniz rule holds. 
   	
  	\section{General definitions and notational conventions}
  	\label{sec:definitions}
  	  	
  	The term \textit{function}  denotes the mapping $ f: \mathbb{R} \mapsto \mathbb{R} $ (or in some cases $\mathbb{C} \mapsto \mathbb{C}$). 
  	The notation $f(x)$ is used to refer to the value of the function at the point \textit{x}.
    The term \textit{operator}  denotes the mapping from one functional expression to another.
  	The symbol \fclass{C}{0} denotes the class of functions, which are continuous. 
  	The symbol \fclass{C}{n} -- the class of \textit{n}-times differentiable functions where  $n  \in \mathbb{N}$.
  	Square brackets are used for the arguments of operators, while round brackets are used for the arguments of functions.	
    $Dom[f]$ denotes the domain of definition of the function $f(x)$.
 	\begin{definition}
 	\label{def:deltas}
 	Let the parametrized difference operators acting on a function $f(x)$ be defined in the following way
 	\begin{align}
 	  	   	\Delta^{+}_{\epsilon} [f](x) & :=  f(x + \epsilon) - f(x) \ecma\\
 	  	   	\Delta^{-}_{\epsilon} [f](x) & :=  f(x) - f(x - \epsilon)  \ecma\\
 	  	   	\Delta^{2}_{\epsilon} [f](x) &:=  f(x + \epsilon) -2 f(x) + f(x - \epsilon) \ecma
 	\end{align}
 	where $\epsilon>0$. The first one we designate as \textit{forward difference} operator, 
 	the second one  as \textit{backward difference} operator and the third one as \textit{2\textsuperscript{nd} order difference} operator.
 	\end{definition}
   	
   	\begin{definition}
   		\label{def:holder}
   		Let \holder{\alpha} be the class of H\"older   $\mathbb{C}^0$  functions of degree $\alpha$, $\alpha \in (0,\, 1)$,
   		That is, we say that $f(x)$ is of class \holder{\alpha} if $\; \forall f (x) \in \holder{\alpha} $ there exist two positive constants 
   		$C, \delta \in \mathbb{R} $ which for given $  x,y \in Dom[ f ]$ such that for $|x-y| \leq \delta$ the following inequality holds
   		\[
   		| f (x) - f (y) |  \leq C |x-y|^\alpha \ecma
   		\]
   		where $| \cdot |$ denotes the norm of the argument.  		
   		Following Mallat and Hwang \cite{Mallat1992} for orders $n>1 $ the definition is extended in the following way: 
   		 H\"older class \holder{n+ \alpha}  designates the class of \fclass{\, C}{0} functions (of degree $n+\alpha$) for which the inequality 
   		\[
   		| f (x) - f (y) - P_n (x-y) |  \leq C |x-y|^{n +\alpha} \ecma
   		\]
   		holds. $P_n (.)$ designates a real-valued polynomial of degree $n \in \fclass{N}{}$ of the form
   		$
   		P_n (z) = \sum\limits_{k=1}^{n}{ a_k z^k} 
   		$,
   		where $P_0(z) = 0$ and $\alpha \in (0, \, 1]$.
   	\end{definition}
   	\begin{remark}
   	The polynomial $P_n(x)$ can be identified with the Taylor polynomial of order $n$ of $f(x)$ \cite{Prodanov2015}.	
    \end{remark}		
   	Under this definition we will focus mainly on functions for which $0< \alpha \leq 1$. 
   	These functions will be further denoted by the term \textbf{H\"olderian}.
   
    \begin{definition}
    	\label{def:osc}
    	Let the oscillation of a function $f(z)$ in the interval $[x, \ y ]$ be defined as 
    	$
    	\mathrm{osc}_{[x, y] }[f] := \sup\limits_z {[ f]} - \inf\limits_z {[ f]}
    	$
    	 for $ z \in \mathrm{[x, \ y ]} $ or in alternative notation by
    	\[
    	\oscplus{f}{x} := \sup_{\epsilon} {[ f]} (x) - \inf_{\epsilon} {[ f]} (x) 
    	\]
    	in the interval $ [x, x + \epsilon ]$;
    	or as
    	\[
    	\oscmin{f}{x} := \sup_{\epsilon} {[ f]} (x) - \inf_{\epsilon} {[ f]} (x) 
    	\]  
    	in the interval $ [ x - \epsilon, x ]$.
    	Let the oscillation about a point be given according to \cite{Trench2013}:
    	\[
    	\mathrm{osc}[f] (x) : = \llim{\epsilon}{0}{\sup_{[x -\epsilon, x + \epsilon]} {[ f]} -
    		\inf_{[x -\epsilon, x + \epsilon]} {[ f]}} \epnt
    	\] 
    \end{definition}
    
   	\section{Fractional variation and $\alpha$-fractional velocity of functions}
   	\label{sec:frdiff}
    Fractional variation operators have been introduced in a previous work in the following way \cite{Prodanov2015}: 
   	\begin{definition}
   		\label{def:fracvar}
   		Let the \textit{Fractional Variation} operators of order $\beta$ be defined as
   		\begin{align}
   			\label{eq:fracvar1}
   			\fracvarplus {f}{x}{\beta} :=   \frac{ f(x+ \epsilon) - f(x) }{\epsilon ^\beta} \ecma
   			\\
   			\fracvarmin {f}{x}{\beta} :=   \frac{ f(x)- f( x- \epsilon)  }{\epsilon  ^\beta} \ecma
   		\end{align}
   		where  $\epsilon >0$ and $0 < \beta \leq 1 $ are real parameters and $f(x)$ is a function.
   	The symbol  $\fracvarplus { \cdot}{ \cdot}{\beta}$ will be called \textsl{forward fractional variation} operator of order $\beta$ while $\fracvarmin {\cdot}{\cdot}{\beta}$ will be called \textsl{ backward fractional variation} operator.	
   	\end{definition}
   	
   	\begin{remark}
	\label{rem:diff2}
	The following representation of the fractional variation operators is equivalent:
	\begin{align*}
	\fracvarplus {f}{x}{\beta}  = \frac{\Delta^{+}_{\epsilon} [f] (x) }{\epsilon ^\beta}    \ecma
	\\
	\fracvarmin {f}{x}{\beta}  =  \frac{\Delta^{-}_{\epsilon} [f] (x) }{\epsilon ^\beta}   \epnt
	\end{align*}
   	\end{remark}
   	
   	\begin{remark}
   	\label{rem:diff1}
	Following the  notation of Cresson and Gref \cite{Cresson2005, Cresson2011} we denote the usual forward and backward differential maps as:
	\begin{flalign*}
	\diffplus{f}{x}{\epsilon} =\frac{f(x + \epsilon)-f(x)}{\epsilon} \ecma \\
	\diffmin{f}{x}{\epsilon} =\frac{f(x)- f(x - \epsilon)}{\epsilon} \epnt
	\end{flalign*}
	These maps, therefore, can be considered as particular cases of the Fractional variation operators according to Definition \ref{def:fracvar}.
   	\end{remark}

   We define \textsl{forward fractional velocity} and  \textsl{backward fractional velocity}, or in short fractional velocity or {$\alpha$-velocity when discussing a particular order, of order $\alpha$ in the following way:
   	 \begin{definition}[Fractional order velocity]
   	 	\label{def:frdiff}
   	 Let  the \textsl{fractional velocity} of fractional order $\alpha$ be defined as
  	  \begin{align}
  	   \label{eq:fracdiffa}
  	    	\fdiffplus {f}{x}{\alpha} &:= \llim{\epsilon}{0}{\frac{\Delta^{+}_{\epsilon} [f ] (x) }{\epsilon ^\alpha}}  \ecma   	\\
  	    	\fdiffmin {f}{x}{\alpha} &:= \llim{\epsilon}{0}{ \frac{\Delta^{-}_{\epsilon} [f ] (x) }{\epsilon ^\alpha}   }  \ecma
  	  \end{align}
  	  where  $\epsilon >0$ and  $0 < \alpha \leq 1 $ are real parameters and $f(x)$ is function.	
  	  \end{definition}

      \begin{proposition}[Continuous variation]
      	\label{th:fdiff}
   	 	If   	 $f^{\prime}(x)$  is continuous in the interval $x \in [x, x+ \epsilon]$ then 
   	\begin{flalign*}
   	\fdiffplus {f}{x}{\beta} =\frac{1}{\beta} \llim{\epsilon}{0}{ \epsilon ^{1-\beta}  f ^{\prime}(x + \epsilon)  }   \ecma
   	\\
   	\fdiffmin {f}{x}{\beta} = \frac{1}{\beta} \llim{\epsilon}{0}{ \epsilon ^{1-\beta}  f ^{\prime}(x - \epsilon)  }   \epnt
   	\end{flalign*}
   	 \end{proposition}
   	 \begin{proof}
   	 	By previous work   \cite{Prodanov2015}.
   	 \end{proof}
   	 The result is repeated for comparison with the expressions for fractional co-variation.
   	 
    \begin{remark} 
   	\label{prop:dualvar}
   	In contrast to the restrictive treatment of the integral-ordered derivatives, we will not require equality of the forward and backward $\alpha$-velocities.
   	In such way it will be possible to handle symmetry breaking, which is important for physical applications, such as scale relativity where explicit symmetry breaking is admitted \cite{Nottale2011}.
   	
   	Indeed, if $f(x)$ is defined in $[ x - \epsilon, x+\epsilon]$ and the limit
   	\[ 
   	\llim{\epsilon}{0}{\dfrac{\Delta_{\epsilon} ^2 [f] (x)}{\epsilon^\beta} } \neq 0
   	\] 
   	exists then  $\fdiffplus {f}{x}{ \beta} \neq \fdiffmin {f}{x}{ \beta}$.
   	The proof of the statement is detailed in \cite{Prodanov2015}.  
    \end{remark}
  	  \begin{example}
  	  	We will calculate the forward and backward fractional variations of order $1/2$ of $f(x)= \sqrt{x}$
  	  	about x=0.  
  	  	\[
  	  	\Delta^{+}_\epsilon f(x)= \sqrt{x+ \epsilon} - \sqrt{x} =	\frac{ \left( \sqrt{x+ \epsilon} - \sqrt{x} \right) . \left(\sqrt{x+ \epsilon} + \sqrt{x}  \right)  }{\sqrt{x+ \epsilon} + \sqrt{x} } =
  	  	\]
  	  	\[
  	  	\frac{ x+ \epsilon - x}{\sqrt{x+ \epsilon} + \sqrt{x}}= \frac{ \epsilon }{\sqrt{x+ \epsilon} + \sqrt{x}} 
  	  	\]
  	  	Then for $x=0$ $\fracvarplus{\sqrt{x}}{x}{1/2} = 1$.
  	  	In a similar manner
  	  	\[
  	  	\Delta^{-}_\epsilon f(x)= \sqrt{x} - \sqrt{x- \epsilon}  =	  \frac{ \left( \sqrt{x} -  \sqrt{x - \epsilon}   \right) . \left(\sqrt{x- \epsilon} + \sqrt{x}  \right)  }{\sqrt{x- \epsilon} + \sqrt{x} } =
  	  	\]
  	  	\[
  	  	\frac{ x - x +\epsilon}{\sqrt{x- \epsilon} + \sqrt{x}}= \frac{ \epsilon }{\sqrt{x- \epsilon} + \sqrt{x}} 
  	  	\]
  	  	Then for $x=0$ $\fracvarmin{\sqrt{x}}{x}{1/2} = - i$.
  	  	
  	  	The same result can be computed using Proposition \ref{th:fdiff}.
  	  	In this case we have 
  	  	\begin{flalign*}
  	  	\fdiffplus{\sqrt{x}}{0}{1/2} &= 2 \llim{\epsilon}{0}{}  \epsilon^{1-1/2} \frac{1}{2   \sqrt{0+ \epsilon}} =1 \\
  	  	\fdiffmin{\sqrt{x}}{0}{1/2} &= 2 \llim{\epsilon}{0}{}  \epsilon^{1-1/2} \frac{1}{2   \sqrt{0 - \epsilon}} = -i 
  	  	\end{flalign*}
  	  	
  	  \end{example}

   \begin{theorem}
   	\label{th:hosc}
   	Let $f(x)$ be defined and finite in the interval $[x, \ y]$.
   	Then 
   	\[
   	| f (x) - f (y) |  \leq \mathrm{osc}_{ |x-y|} [f] \epnt
   	\]
   \end{theorem}
   \begin{proof}
   	The proof follows from the definition of the local oscillation operator.
   	Let $y= x+ \epsilon$. We will consider two cases.
   	\begin{description}
   		\item[Case 1] Let $f(x+ \epsilon) \geq f(x)$. 
   		Since $f(x+ \epsilon) \leq \sup{[ f]} (x)$ and
   		$f(x) \geq  \inf{[ f]} (x)$
   		then subtracting the inequalities leads to
   		\[
   		f(x+ \epsilon) - f(x) \leq \oscplus{f}{x} \epnt
   		\]
   		\item[Case 2] Let $f(x+ \epsilon) \leq f(x)$.
  		Since $f(x) \leq \sup{[ f]} (x)$ and
  		$f(x + \epsilon) \geq  \inf{[ f]} (x)$
  		then subtracting the inequalities leads to
  		\[
  		-f(x+ \epsilon) + f(x) \leq \oscplus{f}{x}   \epnt
  		\]
  		Therefore, 
  		\[
  			| f (x +\epsilon ) - f (x) |  \leq 	\oscplus{f}{x} \epnt
  		\]
   	\end{description}
   \end{proof}
   \begin{corollary}
   	\label{corr:mono1}
   	Let $f(z)$ be monotone in $[x,y]$. 
   	Then 
   	\[
   	| f (x ) - f (y) |  = \mathrm{osc}_{ [x, y]} [f] \epnt
   	\]
   \end{corollary}
   The proof follows immediately from the definitions and  Theorem \ref{th:hosc}.
   \begin{lemma}[Limit of the oscillatory quotient]
   	\label{th:limosc}
   	 \begin{flalign*}
   	\llim{\epsilon}{0}{ }  \frac{ \oscplus{f}{x} }{ \epsilon ^\beta} = \left| \fdiffplus {f}{x}{ \beta} \right| \ecma \\
   	\llim{\epsilon}{0}{ }  \frac{ \oscmin{f}{x} }{ \epsilon ^\beta} = \left| \fdiffmin {f}{x}{ \beta} \right| \epnt
   	 \end{flalign*}
   \end{lemma}
   \begin{proof}
   	 Let $f(x)$ be monotone in $[x, x+ \epsilon]$ then the result follows trivially.
   	 Since $\epsilon$ then is arbitrary there is an interval $[x, x + \epsilon^\prime]$, $ \epsilon^\prime \leq \epsilon$, such that $f(x)$ is monotone inside. 
   	 The backward case can be proved in an identical manner.
   	 This completes the proof. 
   \end{proof}
   \begin{theorem}[Bounds of oscillation theorem]
   	\label{th:hosc2}
   	Let $f(x) \in \holder{\alpha }$, $\alpha \in (0, 1]$ in the interval $[x, x +\epsilon]$.
   	Then 
   	\[
   	0 \leq c \; \epsilon^\alpha \leq \oscplus{f}{x}  \leq C \epsilon^\alpha 
   	\]
   	where $c$ and $C$ are positive real constants.
   \end{theorem}
   \begin{proof} 
   	There are two cases to be considered.
   	\begin{description}
   		\item[Case 1] Let $f(x)$ be monotone. 
   		
	   		Then from the definition of H\"older functions 
	   		and Theorem \ref{th:hosc} it follows that
	   		\[ 
		   	\oscplus{f}{x}\leq C \epsilon^\alpha \epnt
	   		\]
	   		On the other hand, 
	   		$ 		\oscplus{f}{x} \geq B   		$, where $B$ is a non-negative constant.
	   		Dividing all inequalities by the positive amount $\epsilon^\alpha$ gives
	   		\[ 
	   		\frac{B}{\epsilon^\alpha} \leq \frac{ \oscplus{f}{x} }{\epsilon^\alpha} 
	   		\leq C  \epnt
	   		\]
	   		Since $\epsilon$ can be made arbitrarily small then $B= c \; \epsilon^\alpha$ in order to keep
	   		$\frac{ \mathrm{osc}_{\epsilon} [f] }{\epsilon^\alpha}$ finite.
   		\item[Case 2] Let $f(x)$ be non-monotone. 
   		
   			Then if we can split the interval in a collection of touching intervals where $f(x)$ is monotone.
   			\[
   			[x, x+\epsilon] \backslash \{x+\epsilon\}= \bigcup\limits_{x_1=x}^{x_n=x+\epsilon} [x_i, x_{i+1})  \epnt
   			\]
   			Then, the same reasoning applies for every such interval and
   			\[ 
   			\mathrm{osc}_{[x_i, x_{i+1}]} [f] \leq C_i \Delta x_i^\alpha \leq C_i \epsilon^\alpha \epnt
   			\]
   			Therefore, 
   			\[ 
   			\oscplus{f}{x}  \leq \max{(C_i)} \; \epsilon^\alpha \epnt
   			\]
   			Then the lower bound can be established by a similar reasoning as in the previous case.
   			Then
   			\[ 
   			\mathrm{osc}_{[x_i, x_{i+1}]} [f] \geq c_i \Delta x_i^\alpha \geq k_i \epsilon^\alpha 
   			\]
   			where $ k_i \leq \frac{c_i x_i^\alpha}{\epsilon^\alpha}$.
   			Therefore, 
   			$ \oscplus{f}{x} \geq \min{(k_i)} \; \epsilon^\alpha $.
   	\end{description}
   \end{proof}
   \begin{proposition}
   	\label{prop:frvarlim1}
   	Let $f(x ) \in \fclass{H}{\alpha}$ in $[x, x+\epsilon]$ then
   	\fdiffpm{f}{x}{\beta} is unbounded if $\beta> \alpha$, bounded if $\beta = \alpha$ and 0 if $\beta < \alpha$.
   \end{proposition}
   \begin{proof}
   	$
   	\left| \fracvarplus{f}{x}{\beta} \right|  \leq C \; \epsilon^{\alpha -\beta}$. 
   	Therefore, the vanishing and the bounded cases follow.
   	According to Theorem \ref{th:hosc2}
   	\[
   	  \frac{ \oscplus{f}{x} }{\epsilon^{\beta}} \geq c \, \epsilon^{\alpha -\beta } \epnt
   	\] 
   	Therefore, if $ \beta> \alpha$ the limit is unbounded. However by Lemma \ref{th:limosc}
   	\[
   	   	\left|  \fdiffplus{f}{x}{\beta}  \right| = \llim{\epsilon}{0}{}  \frac{ \oscplus{f}{x} }{\epsilon^{\beta}}   \ecma
   	\]
   	therefore, the fractional velocity is unbounded.
   \end{proof}
 The backward case is proven in a similar manner.

 \section{Fractional co-variation of functions}
 \label{sec:cofrdiff}
 The product of two fractional variation operators will be designated as a new operator because of its utility in the expansion of  products of  $\alpha$-velocities.  
 \begin{definition}
 	\label{def:fvar}
 	Let the  co-variation operators in the interval $[x, x+\epsilon]$ be defined as
 	\begin{equation}
 	\label{eq:fracvar}
 	[f,g]_\beta^{\epsilon \pm}(x) :=     \fracvar{f}{x}{\beta/2}{\epsilon \pm} \, \fracvar {g}{x}{\beta/2}{\epsilon \pm} \ecma \\
 	\end{equation}
 	where  $\epsilon >0$ and $0 < \beta \leq 1 $ are real parameters and $f(x)$ and $g(x)$ are functions.
 When $\beta=1$ the index will be omitted to keep notation simple.
 	The symbol  $ \left[ \cdot , \cdot \right]^{+}_{\beta}$ will be called \textsl{forward fractional co-variation} operator while $ \left[ \cdot , \cdot \right]^{-}_{\beta}$ will be called \textsl{ backward fractional co-variation} operator.	
 	
 \end{definition}
 \begin{remark}
 	I can propose the following naming convention. 
 	The order $\frac{1}{2}$ fractional co-variation to be named   \textsl{quadratic variation},
 	the  order $\frac{1}{3}$ -  \textsl{ternary variation},
 	the  order $\frac{1}{4}$ -  \textsl{quaternary variation}, etc.
 \end{remark}

   \begin{definition}
   	\label{def:fvar11}
   	We define the  \textit{fractional co-variation} of two functions $f(x)$ and $g(x)$ about a point $x$ as the limit 
   	\begin{equation}
   	\label{eq:fracvar2}
   	[f,g]_\beta^{ \pm}(x) := \llim{\epsilon}{0}{ }    \fracvar{f}{x}{\beta/2}{\epsilon \pm} \, \fracvar {g}{x}{\beta/2}{\epsilon \pm}  \
   	\end{equation}
   \end{definition}
 \begin{proposition}
 	The fractional co-variation is bi-linear. That is, without making discrimination about the direction of application for  the numbers $\lambda, \mu \in \fclass{R}{}$ and the functions $g(x), f(x), h(x)$ we have 
 	\begin{flalign*}
 	\left[ \lambda \, g, \mu \, f \right]_{\beta} (x) = \lambda \, \mu \, \left[    g,    f \right]_\beta (x) \ecma \\
 	\left[ g + f,  \, h \right]_\beta (x) = \left[ g  ,  \, h \right]_\beta (x)  +\left[ f,  \, h \right]_\beta (x) \epnt
 \end{flalign*}
 \end{proposition}
 \begin{proof}
 	Both properties follows from the linearity of the fractional variation operators. 
 \end{proof}
 
   \begin{proposition}
   	The fractional co-variation is symmetric about its arguments. That is for  functions $g(x), f(x)$  without making discrimination about the direction of application we have 
   	\[
   	\left[ f, g \right]_\beta (x) =\left[ g, f \right]_\beta (x) \epnt
   	\]
   \end{proposition}
   The property follows from the commutativity of multiplication.
   \begin{proposition}
   	Let $f(x) \in \fclass{C}{0}$ and \textit{c} is constant. Then  without making discrimination about the direction of application
   	\[
   	\left[ f, c \right]_\beta (x) =0 \epnt
   	\]
   \end{proposition} 
   The proposition is true since the difference operator yields zero from a constant.
     \begin{figure}[ht]
     	\centering
     	\ifx\JPicScale\undefined\def\JPicScale{1}\fi
     	\unitlength \JPicScale mm
     	\begin{picture}(55,50.21)(0,0)
     	\linethickness{0.4mm}
     	\multiput(21.49,33.4)(0.12,-0.12){139}{\line(1,0){0.12}}
     	\linethickness{0.3mm}
     	\put(5,16.67){\line(1,0){50}}
     	\put(55,16.67){\vector(1,0){0.12}}
     	\linethickness{0.3mm}
     	\multiput(21.6,10)(0.07,40){1}{\line(0,1){40}}
     	\put(21.67,50){\vector(0,1){0.12}}
     	\linethickness{0.3mm}
     	\put(38.33,15.67){\line(0,1){1}}
     	\put(38.33,14.67){\makebox(0,0)[cc]{1.0}}
     	
     	\linethickness{0.3mm}
     	\put(20.67,33.33){\line(1,0){1}}
     	\put(19.67,33.33){\makebox(0,0)[cr]{1.0}}
     	
     	\linethickness{0.4mm}
     	\multiput(38.3,16.7)(1.42,-1.44){5}{\multiput(0,0)(0.12,-0.12){6}{\line(0,-1){0.12}}}
     	\linethickness{0.4mm}
     	\multiput(4.57,50.21)(1.46,-1.45){12}{\multiput(0,0)(0.12,-0.12){6}{\line(1,0){0.12}}}
     	\put(50.05,13.14){\makebox(0,0)[cc]{$\alpha_1$}}
     	
     	\put(16.6,47.34){\makebox(0,0)[cc]{$\alpha_2$}}
     	
     	\put(28.94,33.51){\makebox(0,0)[cc]{$\beta=1$}}
     	
     	\put(38.09,24.68){\makebox(0,0)[cc]{}}
     	
     	\put(25.64,20.11){\makebox(0,0)[cc]{$\infty$}}
     	
     	\put(38.83,33.19){\makebox(0,0)[cc]{0}}
     	
     	\put(39.04,22.87){\makebox(0,0)[cc]{$[f,g]_{\beta}$}}
     	
     	\put(13.62,47.34){\makebox(0,0)[cc]{}}
     	
     	\end{picture}
     	\caption{Fractional co-variation $[f,g]^{\epsilon \pm}(x)$ for 
     		$ f(x) \in \holder{\alpha_1}$ and $ g(x) \in \holder{\alpha_2}$ and $\beta=1$.}
     	\label{fig:qvar}
     \end{figure}
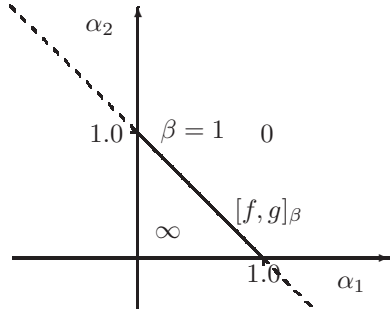
 \begin{theorem}[Bounds of co-variation]
 	\label{th:limitcovar}
 	Let $f(x) \in \holder{\alpha_1}$ and $g(x) \in \holder{\alpha_2}$.
 	Then \llim{\epsilon}{0}{[f,g]_\beta^{\epsilon \pm}(x)} does not vanish only if
 	$ \alpha_1 + \alpha_2 = \beta$.
 	\llim{\epsilon}{0}{[f,g]_\beta^{\epsilon \pm}(x)} is bounded only if 
 	$ \alpha_1 + \alpha_2 \geq \beta$.
 	The result is illustrated in Fig. \ref{fig:qvar}.
 \end{theorem}

 \begin{proof}
 	Since the product of two H\"older functions is a H\"older function of the sum of the grades we have
 	$ \Delta f(x) \Delta g(x) \in \fclass{H}{\alpha_1+ \alpha_2}$.
 	Therefore, by Theorem \ref{th:hosc2} it follows that
 	\[
 	C_1^\prime C_2^\prime \epsilon^{\alpha_1 + \alpha_2} \leq \Delta^{+}_{\epsilon} f(x) \Delta^{+}_{\epsilon}  g(x) 
 	\leq  C_1 C_2 \epsilon^{\alpha_1 + \alpha_2} 
 	\]
 	for certain constants $ C_1^\prime, C_2^\prime ,  C_1,  C_2 $.
 	Therefore, 
 	\[
 	C_1^\prime C_2^\prime \epsilon^{\alpha_1 + \alpha_2 - \beta} \leq  [f,g]^{\epsilon+}_{\beta}
 	\leq  C_1 C_2 \epsilon^{\alpha_1 + \alpha_2 -\beta} \epnt
 	\]
 	Taking the limit provides claimed result.
 	The backward case follows from identical reasoning.
 \end{proof}

   \begin{corollary}
   Let $f(x)$ is uniformly $ \, \fclass{C}{1}$ in the interval $ [x, x+ \epsilon]$ and $g(x) \in \holder{\alpha}, \alpha \leq 1$. 
   Then  without making discrimination about the direction of application
   \[
   [ f, g]_\beta = 0 \epnt
   \]
   \end{corollary}
   \begin{proof}
   	By the last theorem for the exponents we have $1+ \alpha \geq 1$. 
   	This is satisfied for $ \alpha \geq 0$, therefore $ [ f, g]_\beta$ vanishes.
   \end{proof}
 Computation of the fractional co-variation can be related to ordinary derivatives in the following way:
 \begin{theorem}[Limiting case of Fractional co-variation]
 	\label{th:frcovarlim}
 Let $f(x),g(x) \in \fclass{C}{1}$. Then 
 	\begin{flalign*}
 	  \left[f,g \right]^{+}_{\beta} = \frac{1}{\beta} \llim{\epsilon}{0}{} \epsilon^{1-\beta} f^{\prime} (x+ \epsilon) \Delta_{\epsilon}^{+}[g] (x) + \frac{1}{\beta} \llim{\epsilon}{0}{} \epsilon^{1-\beta} g^{\prime} (x+ \epsilon) \Delta_{\epsilon}^{+}[f] (x) \ecma \\
 	  \left[f,g \right]^{-}_{\beta} = \frac{1}{\beta}	\llim{\epsilon}{0}{} \epsilon^{1-\beta} f^{\prime} (x- \epsilon) \Delta_{\epsilon}^{-}[g] (x) + \frac{1}{\beta}	\llim{\epsilon}{0}{} \epsilon^{1-\beta} g^{\prime} (x- \epsilon) \Delta_{\epsilon}^{-}[f] (x) \epnt
 	\end{flalign*}
 \end{theorem}
 
 \begin{proof}
 	The proof follows from application of l'H\^opital's rule to the definition of fractional co-variation.
 	For the forward case we have:
	\begin{flalign*} 
      \left[f,g \right]^{+}_{\beta}   &=   \llim{\epsilon}{0}{} \frac{\left( f (x+ \epsilon) - f(x)  \right)  \left( g (x+ \epsilon) - g(x)   \right) }{\epsilon^\beta}      \\
  & =	 \llim{\epsilon}{0}{} \frac{  f (x+ \epsilon) g (x+ \epsilon)  - f(x) g (x+ \epsilon)    - g(x)  f (x+ \epsilon)  + g(x) f(x)     }{\epsilon^\beta}     
  	\end{flalign*}
  	In the last expression we treat $\epsilon$ as an independent variable and $x$ as a parameter.
  	Application of  l'H\^opital's yields:
  	\begin{flalign*} 
	  \left[f,g \right]^{+}_{\beta}   &	= \llim{\epsilon}{0}{} \frac{\epsilon^{1-\beta}  }{\beta}  \left( f (x+ \epsilon) g (x+ \epsilon)  - f(x) g (x+ \epsilon)    - g(x)  f (x+ \epsilon)  + g(x) f(x)    \right)_{\epsilon}^\prime   \\
	  & = \llim{\epsilon}{0}{} \frac{\epsilon^{1-\beta}  }{\beta} \left(    f^\prime (x+ \epsilon) g (x+ \epsilon) + f (x+ \epsilon) g^\prime (x+ \epsilon) - f(x) g^\prime (x+ \epsilon)    - g(x)  f^\prime (x+ \epsilon)    \right) \\
	  & = \frac{1}{\beta}	\llim{\epsilon}{0}{} \epsilon^{1-\beta} f^{\prime} (x+ \epsilon) \Delta_{\epsilon}^{+}[g] (x) + \frac{1}{\beta}	\llim{\epsilon}{0}{} \epsilon^{1-\beta} g^{\prime} (x+ \epsilon) \Delta_{\epsilon}^{+}[f] (x) \epnt
	\end{flalign*}
	In a similar manner, the backward co-variation can be computed as: 
	\begin{flalign*} 
	\left[f,g \right]^{-}_{\beta}   &=   \llim{\epsilon}{0}{} \frac{\left( f (x- \epsilon) - f(x)  \right)  \left( g (x- \epsilon) - g(x)   \right) }{\epsilon^\beta}      \\
	& =	 \llim{\epsilon}{0}{} \frac{  f (x- \epsilon) g (x- \epsilon)  - f(x) g (x- \epsilon)    - g(x)  f (x- \epsilon)  + g(x) f(x)     }{\epsilon^\beta} 
	\end{flalign*}
	  	In the last expression we treat $\epsilon$ as an independent variable and $x$ as a parameter.
	  	Application of  l'H\^opital's yields:
	\begin{flalign*} 
	\left[f,g \right]^{-}_{\beta}   &	= \llim{\epsilon}{0}{} \frac{\epsilon^{1-\beta}  }{\beta}  \left( f (x - \epsilon) g (x - \epsilon)  - f(x) g (x - \epsilon)    - g(x)  f (x - \epsilon)  + g(x) f(x)    \right)^\prime   \\
	& = \llim{\epsilon}{0}{} \frac{\epsilon^{1-\beta}  }{\beta}  \left( -    f^\prime (x - \epsilon) g (x - \epsilon) - f (x - \epsilon) g^\prime (x - \epsilon) + f(x) g^\prime (x - \epsilon)    + g(x)  f^\prime (x - \epsilon)    \right) \\
	& =  \frac{1}{\beta}	\llim{\epsilon}{0}{} \epsilon^{1-\beta} f^{\prime} (x - \epsilon) \Delta_{\epsilon}^{-}[g] (x) + \frac{1}{\beta}	\llim{\epsilon}{0}{} \epsilon^{1-\beta} g^{\prime} (x - \epsilon) \Delta_{\epsilon}^{-}[f] (x) \epnt
	\end{flalign*}
 \end{proof}
  There is also another reformulation possible.
  \begin{corollary}
   Let $f(x), \ g(x) \in \fclass{C}{1}$ in the interval $[x, x+ \epsilon]$. Then 
   \begin{flalign*}
   \left[f,g \right]^{+}_{\beta} = \frac{1}{\beta} \llim{\epsilon}{0}{} \epsilon \;  f^{\prime} (x+ \epsilon) \, \fracvarplus{g}{x}{\beta}  + \frac{1}{\beta} \llim{\epsilon}{0}{} \epsilon \; g^{\prime} (x+ \epsilon) \, \fracvarplus{f}{x}{\beta}  \ecma \\
   \left[f,g \right]^{-}_{\beta} = \frac{1}{\beta}	\llim{\epsilon}{0}{} \epsilon \; f^{\prime} (x- \epsilon) \, \fracvarmin{g}{x}{\beta}   + \frac{1}{\beta}	\llim{\epsilon}{0}{} \epsilon \;  g^{\prime} (x- \epsilon) \, \fracvarmin{f}{x}{\beta} \epnt
   \end{flalign*}
  \end{corollary}
  \begin{proposition}
  	\label{corr:qvar}
  	For a function $f(x)\in \fclass{C}{1} $ in the interval $[x, x+ \epsilon]$ we have
  	\begin{flalign*}
  	\left[f,f \right]^{+}_{\beta} & =  \frac{2}{\beta}  \llim{\epsilon}{0}{} \epsilon \; f^{\prime} (x+ \epsilon) \, \fracvarplus{f}{x}{\beta} \ecma \\
  	\left[f,f \right]^{-}_{\beta} & = \frac{2}{\beta}	\llim{\epsilon}{0}{} \epsilon \;  f^{\prime} (x- \epsilon) \, \fracvarmin{f}{x}{\beta} \epnt
  	\end{flalign*}	
  	
  \end{proposition}
 
  \begin{example}[Computation of fractional co-variation by Theorem \ref{th:frcovarlim}]
  	
  	We will compute the $\alpha$-velocity of the functions $f(x)=\sqrt{x}$ and $g(x)=^3\sqrt{x}$ for the degree $\alpha=1/2$.
  	\[
	\left[f,g \right]^{+}_{1/2} =\llim{\epsilon}{0}{} 	\sqrt{\epsilon} \, \frac{1}{2}\frac{  ^3\sqrt{x +\epsilon} - ^3\sqrt{x }  }{\sqrt{x+ \epsilon}} +
	  	\sqrt{\epsilon} \, \frac{1}{3}\frac{  \sqrt{x +\epsilon} - \sqrt{x }  }{ ^3\sqrt{(x +\epsilon)^2}   } \epnt 
  	\]
  	For $x=0$ we have 
  	\[
  	\left[f,g \right]^{+}_{1/2} =\llim{\epsilon}{0}{} 	\sqrt{\epsilon} \, \frac{1}{2}\frac{  ^3\sqrt{ \epsilon}    }{\sqrt{ \epsilon}} +
  	\sqrt{\epsilon} \, \frac{1}{3}\frac{  \sqrt{\epsilon}  }{ ^3\sqrt{\epsilon^2}   }  = 0 \epnt
  	\]
  	For $x\neq 0$ since both $f(x)$ and $g(x)$ are continuous  we have 	$[f,g]^{+}_{1/2} = 0$.
  \end{example}
  
 \section{Product rule for fractional variation and $\alpha$-velocity of H\"{o}lderian functions}
 \label{sec:prodrule}
 In order to establish the main result we need to state two technical lemmas.
 The proofs of the lemmas are given only for completeness of presentation. 
 \begin{lemma}[1\textsuperscript{st} Product Lemma]
 	\label{lemma:Deltaplus}
 	For a product of functions $f (x) g (x) $ we have
 	\[
 	\deltaplus{f  g}{x}  =  \deltaplus{f}{x} \,  \deltaplus{g}{x}   + \mathrm{g}(x) \deltaplus{f}{x} + \deltaplus{g}{x}\, \mathrm{f} (x)  \epnt
	 \]
 \end{lemma}
 \begin{proof}
 	Direct calculation shows that
 	\begin{flalign*}
 	\Delta^{+}_\epsilon [\mathrm{f \, g}]  & = \mathrm{f}\left( x+\epsilon\right) \,\mathrm{g}\left( x+\epsilon\right) -\mathrm{f}\left( x\right) \,\mathrm{g}\left( x\right) \\
 	& =\mathrm{f}\left( x+ \epsilon\right) \,\mathrm{g}\left( x+ \epsilon\right) -\mathrm{f}\left( x\right) \,\mathrm{g}\left( x \right) +  \mathrm{f}\left( x+ \epsilon\right) \mathrm{g}(x) - \mathrm{f}\left( x+\epsilon\right) \mathrm{g}(x) \\
 	& =  \mathrm{f}\left( x+\epsilon\right)   \deltaplus{g}{x}  + \mathrm{g}(x)   \deltaplus{f}{x}   \\
 	& = \mathrm{f}\left( x+\epsilon\right) \deltaplus{g}{x}   + \mathrm{g}(x) \deltaplus{f}{x}  + \deltaplus{g}{x} \, \mathrm{f} (x) - \deltaplus{g}{x}  \, \mathrm{f} (x) \\
 	& = \deltaplus{f}{x}  \,  \deltaplus{g}{x}    + \mathrm{g}(x) \Delta^{+}_\epsilon\mathrm{f}   + \deltaplus{g}{x}  \, \mathrm{f} (x)  \epnt
 	\end{flalign*}

 \end{proof}
 
 \begin{lemma}[2\textsuperscript{nd} Product Lemma]
 	\label{lemma:Deltaminus}
 	\[
 	\deltamin{f g}{x}   = - \deltamin{f}{x} \,  \deltamin{g}{x} + \mathrm{g}(x) \deltamin{f}{x} + \deltamin{g}{x}\, \mathrm{f} (x) \epnt
 	\]
 \end{lemma}
 \begin{proof}
 	Direct calculation shows that
 	\begin{flalign*}
 		\Delta^{-}_{\epsilon} [\mathrm{f \, g}]   & = \mathrm{f} \left( x\right) \,\mathrm{g}\left( x\right) - \mathrm{f}\left( x - \epsilon \right) \,\mathrm{g} \left( x -\epsilon\right) \\
 		& = \mathrm{f}\left( x \right)\,\mathrm{g}\left( x\right) -\mathrm{f}\left( x - \epsilon \right) \,\mathrm{g}\left( x-\epsilon\right)   +  \mathrm{f}\left( x - \epsilon\right) \mathrm{g}(x) - \mathrm{f}\left( x -\epsilon\right) \mathrm{g}(x) \\
 		& = \mathrm{f}\left( x- \epsilon \right) \deltamin{g}{x}   + \mathrm{g}(x)   \deltamin{f}{x}    \\
 		& = \mathrm{f}\left( x-\epsilon \right)  \deltamin{g}{x}    + \mathrm{g}(x) \deltamin{f}{x}    +  \deltamin{g}{x} \, \mathrm{f} (x) - \deltamin{g}{x}  \, \mathrm{f} (x) \\
 		& =	- \deltamin{f}{x}   \,   \deltamin{g}{x}    + \mathrm{g}(x) \deltamin{f}{x}     +  \deltamin{g}{x}  \, \mathrm{f} (x) \epnt
 	\end{flalign*}

 \end{proof}
 Having established these results, the product rule for the fractional variation operators and $\alpha$-velocity can be stated in the following Theorem:
 \begin{theorem}[Product rule for Fractional variation]
 	\label{lemma:prod1}
 	Let   $f(x) \in \holder{\alpha_1}$ and $g(x) \in \holder{\alpha_2}$. 
 	Then
 	\begin{align}
 	\fracvarplus {f \, g}{x}{\beta}=  \fracvarplus{f}{x}{\beta} g (x) +  \fracvarplus{g}{x}{\beta}f(x) 
 	+	[f,g]_\beta^{\epsilon+}(x) \ecma \\
 	\fracvarmin{f \, g}{x}{\beta}=  \fracvarmin{f}{x}{\beta} g (x) +  \fracvarmin{g}{x}{\beta}f(x) 
 	-	[f,g]_\beta^{\epsilon-}(x) \epnt
 	\end{align}
 \end{theorem}
 
 \begin{proof}
 	We prove the forward case first. 
 	\begin{description}
 		\item[Forward case] 
 		By Lemma \ref{lemma:Deltaplus} it follows that
 		\[
 		\Delta^{+}_{\epsilon} [\mathrm{f \, g}] (x)  =  
 		\Delta^{+}_{\epsilon}[\mathrm{f}](x) \,  \Delta^{+}_\epsilon [\mathrm{g}](x)   + \mathrm{g}(x) \Delta^{+}_{\epsilon}[\mathrm{f}]  (x) + \Delta^{+}_\epsilon[\mathrm{g}](x) \, \mathrm{f} (x) \epnt 
 		\]
 		Dividing the last quantity to $\epsilon^\beta$ yields
 		\begin{flalign*}
 		\frac{\Delta^{+}_{\epsilon} [\mathrm{f \, g}] (x)}{\epsilon^\beta}  &= \frac{\Delta^{+}_{\epsilon}\mathrm{f} (x)\,  \Delta^{+}_{\epsilon}\mathrm{g} (x) }{\epsilon^\beta} +  \fracvarplus{f}{x}{\beta} g (x) +  \fracvarplus{g}{x}{\beta}f(x)  \\
 		& = [f,g]_\beta^{\epsilon+}(x) + \fracvarplus{f}{x}{\beta} g (x) +  \fracvarplus{g}{x}{\beta}f(x) \ecma 
 		\end{flalign*}	 
 		which completes the proof.
 		\item[Backward case] 
 		The backwards case follows from analogous reasoning.
 		By Lemma \ref{lemma:Deltaminus} it follows that 
 		\[
 		\Delta^{-}_{\epsilon} [\mathrm{f \, g}] (x)  =   
 		- \Delta^{-}_{\epsilon}\mathrm{f} (x)\,  \Delta^{-}_{\epsilon}\mathrm{g} (x)  + \mathrm{g}(x) \Delta^{-}_{-\epsilon}\mathrm{f} (x)  + \Delta^{-}_{\epsilon}\mathrm{g}(x) \, \mathrm{f} (x)  \epnt
 		\]
 		Dividing the last quantity to $ \epsilon^\beta$ yields
 		\begin{flalign*}
 		\frac{\Delta^{-}_{\epsilon} [\mathrm{f \, g}] (x)}{\epsilon^\beta}  &= -\frac{\Delta^{-}_{\epsilon}\mathrm{f} (x) \,  \Delta^{-}_{\epsilon}\mathrm{g} (x) }{\epsilon^\beta}  + \fracvarmin{f}{x}{\beta} g (x) +  \fracvarmin{g}{x}{\beta}f(x)  \\
 		& =-[f,g]_\beta^{\epsilon-}(x) + \fracvarmin{f}{x}{\beta} g (x) +  \fracvarmin{g}{x}{\beta}f(x) \ecma
 		\end{flalign*}
 	
 	\end{description}
 	which completes the proof.
 \end{proof}
  \begin{corollary}[Limiting case]
  For the functions $f(x)$ and $g(x)$
  \begin{eqnarray}
 \fdiffplus{ [f \, g]  }{x}{\beta}    =   \fdiffplus{ f}{x}{\beta}  g (x) +   \fdiffplus{ g}{x}{\beta} f(x) 	+	[f,g]^{+}_\beta(x) \ecma \\
 \fdiffmin{ [f \, g]  }{x}{\beta}    =   \fdiffmin{ f}{x}{\beta}  g (x) +   \fdiffmin{ g}{x}{\beta} f(x) 	-	[f,g]^{-}_\beta(x) \epnt
  \end{eqnarray}
  \end{corollary}
    	This corollary confirms a recent result by Tarasov \cite{Tarasov2013} who affirms violation of the Leibniz's rule for fractional derivatives of non-differentiable functions. 
    	Indeed, by Theorem \ref{th:limitcovar} the quantity $ [f,g]^{\pm}_{\beta} $ is not identically zero.
     \begin{proposition}[Square case]
     For the function  $f(x)$ we have
     	\begin{eqnarray}
     	\fdiffplus{ f^2 }{x}{\beta}    =  2 \,  \fdiffplus{ f}{x}{\beta}  f (x) + [f,f]^{+}_\beta(x) \ecma \\
     	\fdiffmin{ f^2  }{x}{\beta}    =  2 \, \fdiffmin{ f}{x}{\beta}  f (x) -  	[f,f]^{-}_\beta(x) \epnt
     	\end{eqnarray}
     \end{proposition}
     
       \begin{corollary}[Quotient case]
       	 For the functions $f(x)$ and $g(x)$
       	\begin{eqnarray}
       	\fdiffplus{   [f / g ]   }{x}{\beta}    & =  \frac{  \fdiffplus{ f}{x}{\beta}  g (x) -  \fdiffplus{ g}{x}{\beta} f(x) -[f,g]^{+}_{\beta} }{g^2(x)}	 \ecma \\
       	\fdiffmin{ [f / g ]  }{x}{\beta}    & =   \frac{\fdiffmin{ f}{x}{\beta}  g (x) -   \fdiffmin{ g}{x}{\beta} f(x) + [f,g]^{-}_{\beta}  }{g^2(x)}	  \ecma
       	\end{eqnarray}
       	where we assume that $g(x) \neq 0$.
       \end{corollary}
       \begin{proof}
       	We will prove the forward case first.
       	Since $g(x) \neq 0$ we have
 \[
  \fdiffplus{ f}{x}{\beta}  \frac{1}{g (x)} +   \fdiffplus{\left[ 1/g \right]  }{x}{\beta} f(x) 	+	[f,1/g]^{+}_\beta(x) \epnt
 \]
 The fractional variation of the quotient function can be calculated as
 \[
\fracvarplus {  1/g  }{x}{\beta} = \frac{1}{\epsilon^\beta} \left(\frac{1}{g (x+ \epsilon)}- \frac{1}{g(x) }  \right) = -  \frac{\fracvarplus {g}{x}{\beta} }{g(x) g(x +\epsilon)} \epnt
 \]
 Taking the limit yields 
 \[
  \fdiffplus{\left[ 1/g \right]  }{x}{\beta} = - \llim{\epsilon}{0}{\frac{\fracvarplus {g}{x}{\beta} }{g(x) g(x +\epsilon)}} = - \frac{\fdiffplus{ g }{x}{\beta} }{g^2(x)} \ecma
  \]
  since by assumption $g(x) \neq 0$.
  Therefore, for the fractional co-variation it follows that
  \[
  [f, 1/g]^{\epsilon+}_{\beta} (x) = - \frac{\fracvarplus {  f  }{x}{\beta/2} \fracvarplus {  g }{x}{\beta/2}}{g (x) g(x + \epsilon)} 
  = - \frac{[f,g]^{\epsilon+}_{\beta}}{g (x) g(x + \epsilon)}
  \]
  and
  \[
  [f, 1/g]^{ +}_{\beta} (x) = - \frac{[f,g]^{+}_{\beta}}{g^2(x)} \ecma
  \]
  respectively.
  
  For the backward case  we have to not that 
   \[
   \fracvarmin {  1/g   }{x}{\beta} = \frac{1}{\epsilon^\beta} \left( \frac{1}{g(x)} - \frac{1}{g (x - \epsilon)  }  \right) = -  \frac{\fracvarmin {g}{x}{\beta} }{g(x) g(x -\epsilon)} \epnt
   \]
   Taking the limit yields 
   \[
   \fdiffmin{\left[ 1/g \right]  }{x}{\beta} = - \llim{\epsilon}{0}{\frac{\fracvarmin {g}{x}{\beta} }{g(x) g(x -\epsilon)}} = - \frac{\fdiffmin{ g }{x}{\beta}}{g^2(x)  }  \ecma
   \]
   since by assumption $g(x) \neq 0$.
    Therefore, for the fractional co-variation it follows that
    \[
    [f, 1/g]^{\epsilon-}_{\beta} (x) = - \frac{\fracvarmin{  f  }{x}{\beta/2} \fracvarplus {  g }{x}{\beta/2}}{g(x) g(x -\epsilon)} 
    = - \frac{[f,g]^{\epsilon-}_{\beta}}{g(x) g(x -\epsilon)}
    \]
    and taking the limit finally yields
    \[
    [f, 1/g]^{ -}_{\beta} (x) = - \frac{[f,g]^{-}_{\beta}}{g^2(x)} \ecma
    \]
    respectively.
  \end{proof}
  The proof established also the following additional results:
  \begin{proposition}
  For the functions $f(x)$ and $g(x) \neq 0 $ we have
  	\[
  	[f, 1/g]^{ \pm}_{\beta} (x) = - \frac{[f,g]^{\pm}_{\beta}}{g^2(x)} \epnt
  	\]
  \end{proposition}
  \begin{proposition}
  For the function   $g(x) \neq 0 $ we have 
  	\[
  	 \fdiffpm{\left[ 1/g \right]  }{x}{\beta} =  - \frac{\fdiffpm{ g }{x}{\beta} }{g^2(x)} \epnt
   \]
  \end{proposition}
      \begin{example}[Direct calculation of quadratic variation]
      	We will calculate the quadratic variation of $f(x)= \sqrt{x}$. Then
      	\[
      	\Delta^{+}_\epsilon f(x)= \sqrt{x+ \epsilon} - \sqrt{x} =	\frac{ \left( \sqrt{x+ \epsilon} - \sqrt{x} \right) \left(\sqrt{x+ \epsilon} + \sqrt{x}  \right)  }{\sqrt{x+ \epsilon} + \sqrt{x} } =
      	\]
      	\[
      	\frac{ x+ \epsilon - x}{\sqrt{x+ \epsilon} + \sqrt{x}}= \frac{ \epsilon }{\sqrt{x+ \epsilon} + \sqrt{x}} 
      	\]
      	Therefore 
      	\[ 
      	[ \sqrt{x}, \sqrt{x} ]^{\epsilon+} = \frac{\left(\frac{ \epsilon }{\sqrt{x+ \epsilon} + \sqrt{x}}  \right)^2 }{\epsilon} =
      	\frac{ \epsilon }{ \left( \sqrt{x+ \epsilon} + \sqrt{x} \right)^2} 
      	\]
      	Therefore, 
      	\[
      	[ \sqrt{x}, \sqrt{x} ]^{+} = \left\{
      	\begin{array}{ll}   
      	1  ,& x=0  \\      
      	0  ,&  x > 0
      	\end{array}
      	\right.
      	\]
      \end{example}
      \begin{example}[Calculation of quadratic variation from Theorem \ref{th:frcovarlim}]
      	\[
      	[ \sqrt{x}, \sqrt{x} ]^{+} = \llim{\epsilon}{0}{} 2 \; \frac{1}{2 \, \sqrt{x +\epsilon}}\left( \sqrt{x +\epsilon} -\sqrt{x}\right) = 1-  \llim{\epsilon}{0}{} \frac{\sqrt{x}}{\sqrt{x +\epsilon}}  
      	\]
      	Therefore, 
      	\[
      	[ \sqrt{x}, \sqrt{x} ]^{+} = \left\{
      	\begin{array}{ll}   
      	1  ,& x=0  \ecma \\      
      	0  ,&  x > 0 \epnt
      	\end{array}
      	\right.
      	\]
      \end{example}
 Combining the evaluation of the product rules leads to the following results
\begin{proposition}
For functions $f(x)$ and $g(x)$
\begin{flalign*}
\fdiffplus{ \left[ f \ g \right]  }{x}{\beta}  = & \frac{1}{\beta} \llim{\epsilon}{0}{} \epsilon^{1-\beta}
\left( g^\prime(x + \epsilon) f(x)+ g (x) f^\prime(x + \epsilon) \right) + \\
& \frac{1}{\beta} \llim{\epsilon}{0}{} \epsilon^{1-\beta} \left(  g^\prime(x + \epsilon) \deltaplus{f}{x}  + f^\prime(x + \epsilon) \deltaplus{g}{x}  \right) 
\end{flalign*}
\end{proposition}    
\begin{proposition}
   For functions $f(x)$ and $g(x)$
\begin{flalign*}   
    \fdiffmin{ \left[ f \ g \right]  }{x}{\beta} = &\frac{1}{\beta} \llim{\epsilon}{0}{}  \epsilon^{1-\beta}
    \left( g^\prime(x - \epsilon) f(x) + g (x) f^\prime(x - \epsilon)  \right) - \\
   & \frac{1}{\beta} \llim{\epsilon}{0}{}  \epsilon^{1-\beta} \left(  g^\prime(x - \epsilon) \deltamin{f}{x} + f^\prime(x - \epsilon) \deltamin{g}{x}  \right) 
 \end{flalign*}
\end{proposition}    
        
 \subsection{Relation to the Leibniz's rule}
 \label{sec:leibitz}

 Under the particular notation in Remark \ref{rem:diff1} the product rule restated in the  following way:
 \begin{proposition}
 	\label{lemma:dprod}
 	Under the notation used in this work, the product rule for ordinary differential quotients can be expressed by
 	\begin{align}
 	\diffplus{f \, g}{x}{\epsilon}= \diffplus{f}{x}{\epsilon} g (x) + \diffplus{g}{x}{\epsilon} f(x) 
 	+	[f,g]_\epsilon^{+}(x) \ecma \\
 	\diffmin{f \, g}{x}{\epsilon}= \diffmin{f}{x}{\epsilon} g (x) + \diffmin{g}{x}{\epsilon} f(x)  
 	- [f,g]_\epsilon^{-}(x) \epnt
 	\end{align}
 \end{proposition}
 \begin{proof}
 The result has been stated in \cite{Cresson2005, Cresson2011}. 
 However, we can consider it as a specialization of the $\alpha$-fractional velocity product rule (Theorem \ref{lemma:prod1}) for the case $\alpha=1$.
 \end{proof}
 If we take the limit we arrive at the usual Leibniz rule for derivatives:
  \begin{proposition}
   	\begin{align}
   	\diffplus{f \, g}{x}{ }= \diffplus{f}{x}{ } g (x) + \diffplus{g}{x}{ } f(x) 
   	+	[f,g]^{+}(x) \\
   	\diffmin{f \, g}{x}{ }= \diffmin{f}{x}{ } g (x) + \diffmin{g}{x}{ } f(x)  
   	- [f,g]^{-}(x)
   	\end{align}
   	For functions  in $\mathbb{C}^1$ with uniformly continuous derivatives taking the limit to zero yields the usual Leibniz rule since
   	$ [f,g]^{-}(x) = [f,g]^{+}(x)=0$. 
   \end{proposition}

 \section{Discussion}
 \label{sec:disc}
 
    The theory of local fractional derivatives is still immature and there are few established results with certainty  \cite{Adda2001, Babakhani2002, Adda2004, Adda2005, Kolwankar2001, Chen2010, Adda2013}.
 	Historically, the first non-integral order local definition has been proposed by Cherbit \cite{Cherbit1991}  as the notion of "$\alpha$-velocity" in the study of fractals as
 	\[
 	X^{(\alpha)}:= \llim{s}{0}{}\frac{X(t+s)- X(t)}{s^\alpha} \ecma
 	\]
 	where $\alpha \in (0, 1]$ for the trajectory $X(t)$ as a function of time.
 	The concept implicitly implied continuity. 
 	Chen et al.  used the term "difference-quotient" based local fractional derivative \cite{Chen2010}.
  	It should be clear that the terms  $\alpha$-velocity (Cherebit), "difference-quotient" based local fractional derivative (Chen et al.) and  "$\alpha$-derivative" (Ben Adda and Cresson) are equivalent.
  	Regrettably, in the current literature there is still no unification of notations and terms, which provides space for confusion.
  	Properties of this entity are qualitatively different from the properties of ordinary derivatives and the integral fractional derivatives. 
  	Therefore, its study has its own merits. 
  	For example, it has very few non-zero or non-divergent values therefore, it has more of a supplementary character compared to integral derivatives.


\section*{Acknowledgments}
The work has been supported in part by a grant from Research Fund - Flanders (FWO), contract number 0880.212.840.

\bibliographystyle{plain}  
\bibliography{qvar}

%

\end{document}